\documentclass[12pt, reqno]{amsart}
\usepackage{amsmath, amsthm, amscd, amsfonts, amssymb, graphicx, color}
\usepackage[bookmarksnumbered, colorlinks, plainpages]{hyperref}

\newtheorem{thm}{Theorem}[section]
\newtheorem{lem}{Lemma}[section]

\newtheorem{cor}{Corollary}[section]

\newtheorem{defn}{Definition}[section]

\numberwithin{equation}{section}

\begin{document}

\begin{center}
{\large{\textbf{CHARACTERISTICS OF ALMOST CONFORMAL RICCI SOLITONS ON SASAKIAN MANIFOLD}}}
\end{center}
\vspace{0.1 cm}
\begin{center}By\end{center}
\begin{center}

{Dipen~~Ganguly\footnote{The first author is the corresponding author.}, Nirabhra~~Basu$^2$ and Arindam~~Bhattacharyya$^3$}
\end{center}
\vskip 0.3cm
\begin{center}
$^{1,3}$Department~of~Mathematics\\
Jadavpur~University,\\
Kolkata-700032,~India.\\
E-mail: dipenganguly1@gmail.com\\
E-mail: bhattachar1968@yahoo.co.in\\
$^{2}$Department~of~Mathematics\\
The Bhawanipur~Education~Society~College,\\
Kolkata-700020,~India.\\
E-mail: nirabhra.basu@thebges.edu.in
\end{center}
\vskip 0.3cm
\begin{center}
\textbf{Abstract}
\end{center}\par
\medskip
In this paper, we characterize the potential function $f$ of the almost conformal gradient Ricci soliton on a Sasakian manifold in terms of the non-dynamical scalar field $p$ and deduce the necessary condition for the potential function $f$ to be constant. Furthermore, a relation between $\lambda$ and the potential function $f$ has been established. Finally, we prove a sufficient condition for an almost conformal Ricci soliton to be an almost conformal gradient Ricci soliton and also a characterization of the soliton in terms of shrinking, steady or expanding has been done.\par
\medskip
\begin{flushleft}
\par
\end{flushleft}
\medskip
\begin{flushleft}
 \par
\end{flushleft}
\medskip
\medskip
\medskip
\section{\textbf{Introduction}}
G. Perelman's \cite{P} solution to prove Poincare conjecture with the help of Ricci flow has crossed almost fifteen years and then after many mathematicians extended the study with different ideas like conformal Ricci flow, conformal Ricci soliton, almost Ricci soliton etc, though Ricci flow was introduced by R. S. Hamilton in 1982. The Ricci flow equation follows as
\begin{equation}
\frac{\partial g}{\partial t}=-2Ric.
\end{equation}
\par
A smooth manifold $M$ equipped with a Riemannian metric $g$ is said to be a Ricci soliton, if for some constant $\lambda$, there exist a smooth vector field $X$ on $M$ satisfying
\begin{equation}
Ric+\frac{1}{2}\mathcal{L}_{X}g=\lambda g,
\end{equation}
where $\mathcal{L}_{X}$ denotes the Lie derivative and $Ric$ is the Ricci tensor. The Ricci soliton is called shrinking if $\lambda >0$, steady if $\lambda =0$ and expanding if $\lambda <0$.
Ricci soliton can also be viewed as natural generalization of Einstein metric which moves only by an one-parameter group of diffeomorphisms and scaling \cite{H2,CC}. Again a Ricci soliton is called a gradient Ricci soliton \cite{H1} if the concerned vector field X in the equation  is the gradient of some smooth function $f$ and the equation becomes
\begin{equation}
Ric + \nabla(\nabla f) = \lambda g,
\end{equation}
where $\lambda$ is a constant.
\par
\medskip
The concept of Ricci almost soliton was propounded by S. Pigola et al., \cite{PR} as a natural extension of the gradient Ricci soliton. Later on a substantial amount of study has been done by various authors like R. Sharma \cite{S}, A. Barros et al., \cite{BBR} and many others on Ricci almost soliton. The Ricci almost soliton equation on a Riemannian manifold $(M,g)$ is given by
\begin{equation}
2Ric +\mathcal{L}_{X}g = 2\lambda g,
\end{equation}
 where $\lambda$ is a smooth function on $M$. In particular if $\lambda$ is constant, then the equation $(1.4)$ is same as the Ricci soliton equation $(1.2)$. Again if for some smooth function $f$, the vector field $X$ is the gradient of the function $f$, then the soliton is called a gradient Ricci almost soliton. Therefore by taking $X=\nabla f$ in the previous equation, we have the equation of the gradient Ricci almost soliton as
\begin{equation}
Ric + \nabla(\nabla f) = \lambda g,
\end{equation}
 where $\lambda$ is a real-valued smooth function on $M$.
 \par
\medskip
A. E. Fischer \cite{F}, in 2004, came up with a new idea of conformal Ricci flow which is a modified version of the Hamilton's Ricci flow equation. The unit volume constraint has very significant role in the classical theory of Hamilton's Ricci flow equation but in the case of the conformal Ricci flow equation the main difference is the scalar curvature constraint. For this reason this new Ricci flow equation is defined as the conformal Ricci flow as the scalar curvature constraint plays pivotal role in conformal geometry. Let $(M,g)$ be a closed connected oriented $n$-dimensional, $n\geq 3$, smooth manifold; then  the conformal Ricci flow equation on $(M,g)$ is given by
\begin{equation}
\frac{\partial g}{\partial t}+2(Ric+\frac{g}{n})=-pg,
\end{equation}
\begin{equation}
r(g)=-1, \nonumber
\end{equation}
where p is the non dynamical (time dependent) scalar field. This geometric evolution equation is analogous to the Navier-Stokes equation, which is a very well-known topic of fluid mechanics.
\par
\medskip
 Recently In 2015, N. Basu et al., \cite{B1}, have introduced the concept of conformal Ricci flow which satisfies the conformal Ricci flow equation. The equation follows as
\begin{equation}
\mathcal{L}_{X}g+2Ric=[2\lambda -(p+\frac{2}{n})]g.
\end{equation}
\par
 T. Dutta et al., \cite{TD} have introduced the notion of almost conformal gradient Ricci soliton and established some curvature identities for almost conformal gradient Ricci soliton. The equation of an almost conformal gradient Ricci soliton is given by
 \begin{equation}
\mathcal{L}_{X}g+2Ric=[2\lambda -(p+\frac{2}{n})]g,
\end{equation}
where $\lambda$ is a real-valued smooth function on $M$ and the vector field $X$ is the gradient of some smooth function.\par
\medskip
D. G. Prakasha et. al., \cite{DG}, in 2016, studied Ricci solitons and $\eta$-Ricci solitons within the framework of para-Sasakian geometry. In 2011, C. He and M. Zhu \cite{HZ} showed that a Sasakian metric satisfying the gradient Ricci soliton equation is necessarily Einstein. After that, N. Basu and A. Bhattacharyya \cite{B2} have found out the nature of the potential function $f$ of almost Ricci solitons in Sasakian manifold. Motivated by the above studies, here we have characterized the potential function $f$ for a Sasakian manifold satisfying the almost conformal gradient Ricci soliton equation. This paper is organised as follows: first we discuss some basic preliminaries of Sasakian manifolds in section-2. Finally in section-3, our main results on Sasakian manifolds admitting almost conformal gradient Ricci solitons, have been established.
\medskip
\medskip
\medskip
\section{\textbf{Basic Preliminaries}}\par
\medskip
Let $M^{2n+1}$ be a $(2n+1)$ dimensional smooth manifold and let there exist a $(1,1)$ tensor field $\phi$, a vector field $\xi$ and a global 1-form $\eta$ on M such that
\begin{equation}
\eta(\xi)=1
\end{equation}
and
\begin{equation}
\phi^2 X=-X + \eta(X)\xi,
\end{equation}
for any vector field X on $M^{2n+1}$, then we say that M has an almost contact structure $(\phi,\xi,\eta)$. The manifold $M^{2n+1}$ equipped with this almost contact structure $(\phi,\xi,\eta)$ is called an almost contact manifold \cite{Bla}. \par
\medskip
Here $\xi$ is called the Reeb vector field or the characteristic vector field. For an almost contact structure $(\phi,\xi,\eta)$ the following relations hold; \par $\phi\circ\xi=0$ and $\eta\circ\phi=0$.\par
\medskip
\begin{lem}
  [\cite{Bla}] \emph{Every almost contact structure $(\phi,\xi,\eta)$ on a manifold $M^{2n+1}$ admits a Riemannian metric $g$ satisfying;}
  \begin{align}
 1) g(\phi X,\phi Y) &=g(X,Y)-\eta(X)\eta(Y), \nonumber\\
 2) \eta(X) &=g(X,\xi).
  \end{align}
\end{lem}
\medskip
 Then the metric $g$ is called compatible with the almost contact structure $(\phi,\xi,\eta)$ and the manifold $M^{2n+1}$ with the almost contact metric structure $(\phi,\xi,\eta,g)$ is called an almost contact metric manifold.\par
\medskip
An almost contact manifold $M^{2n+1}$ together with the almost contact structure $(\phi,\xi,\eta)$ is said to be a Sasakian manifold or a normal contact metric manifold if;
\begin{equation}
[\phi,\phi](X,Y)+2d\eta(X,Y)\xi=0, \nonumber
\end{equation}
where; $[\phi,\phi]$ denotes the Nijenhuis torsion tensor field of $\phi$ and is given by;
\begin{equation}
[\phi,\phi](X,Y)=\phi^2[X,Y]+[\phi X,\phi Y]-\phi([\phi X,Y])-\phi([X,\phi Y]). \nonumber
\end{equation}
Or equivalently, an almost contact metric manifold $M^{2n+1}$ with the structure $(\phi,\xi,\eta,g)$ is called a Sasakian manifold [see\cite{BG}] if ;
\begin{equation}
(\nabla_{X}\phi)Y=g(X,Y)\xi -\eta(Y)X, \nonumber
\end{equation}
where $\nabla$ is the Levi-Civita connection on $M^{2n+1}$.
\begin{thm}
  [\cite{BG}] \emph{A contact manifold $M^{2n+1}$ with the contact metric structure $(\phi,\xi,\eta,g)$ is a Sasakian manifold if and only if, the Killing vector field $\xi$ satisfies the relation;}
\begin{equation}
R(Y,\xi)Z=\eta(Z)Y-g(Y,Z)\xi,
\end{equation}
\emph{$\forall Y,Z\in TM$ and $R$ is the Riemannian curvature tensor of the manifold $M^{2n+1}$.}
\end{thm}
 \par
\medskip
Next, let us consider the distribution $\mathcal{D}\subset TM$ defined by
\begin{equation}
\eta(Y)=g(Y,\xi)=0
\end{equation}
 Then as $\eta$ is a contact 1-form, the distribution $\mathcal{D}$ is nowhere integrable and for any $Y\in TM$ we have $\phi(Y)=\nabla_{Y}\xi\in\mathcal{D}$ since $\xi$ is a Killing vector field. Now, applying the inner product with respect to $Y$, on both sides of the equation (2.4), we get
 $$R(Y,\xi ,Z,Y)=\eta(Z)g(Y,Y)-g(Y,Z)g(\xi ,Y).$$
 Then, using the distribution condition in equation(2.5) this reduces to
 $$R(Y,\xi ,Z,Y)=\eta(Z)g(Y,Y)$$
 and applying equation(2.3) in this we get
 $$R(Y,\xi ,Z,Y)=g(Z,\xi)g(Y,Y).$$
 Now, choosing an orthonormal basis vector for $Y=e_i$ and taking summation over the index $i$ we get
 \begin{equation}
 Ric(\xi ,Z)=2ng(\xi ,Z)
 \end{equation}
So, now we use the above results to prove our main theorem.
\medskip
\medskip
\medskip
\section{\textbf{Main Results}}\par
\medskip
\medskip
In this section we prove our two main theorems on almost conformal gradient Ricci soliton. First we consider a sasakian manifold with the distribution $\mathcal{D}\subset TM$ satisfying the almost conformal gradient Ricci soliton.
\medskip
\begin{thm}
Let $(M^{2n+1},g)$ be a Sasakian manifold satisfying the almost conformal gradient Ricci soliton equation and $\mathcal{D}\subset TM$ be the distribution. Then if $\lambda =\frac{p}{2}$ i.e; if $\lambda$ is half of the conformal pressure $p$; the potential function $f$ is constant, for non-zero function $\lambda$.
\end{thm}
\begin{proof}
So, here for our concerned Sasakian manifold $(M^{2n+1},g)$ satisfying the almost conformal Ricci soliton equation, we get from equation(1.8)
$$\mathcal{L}_{X}g(\xi ,W)+2Ric(\xi ,W)=[2\lambda -(p+\frac{2}{2n+1})]g(\xi ,W).$$
Then using the value of $Ric$ from equation(2.6) we get
$$\mathcal{L}_{X}g(\xi ,W)=[2\lambda -(p+\frac{2}{2n+1})-4n]g(\xi ,W)=(2\lambda -p+k)g(\xi ,W),$$
where $k=-\frac{2}{2n+1}-4n$ is a fixed constant for this manifold. This implies
$$(2\lambda -p+k)g(\xi ,W)=Xg(\xi ,W)-g([X,\xi],W)-g(\xi ,[X,W])$$
$$=\nabla _Xg(\xi ,W)-g(\nabla _X\xi ,W)+g(\nabla _\xi X,W)-g(\xi ,\nabla _XW)+g(\xi ,\nabla _WX)$$
$$=(\nabla _Xg)(\xi ,W)+g(\nabla _\xi X,W)+g(\xi ,\nabla _WX).$$
Now $\nabla$ being a metric connection $(\nabla _Xg)(\xi ,W)$ vanishes and hence we have
$$g(\nabla _\xi X,W)+g(\xi ,\nabla _WX)=(2\lambda -p+k)g(\xi ,W).$$
Putting $W=\xi$ in the above and simplifying we get
\begin{equation}
g(\nabla _\xi X,\xi)=(\lambda -\frac{p}{2}+\frac{k}{2})g(\xi ,\xi).
\end{equation}
Again, $\xi$ being a Killing vector field, we have
$$(\mathcal{L}_{\xi}(\mathcal{L}_{X}g))(Y,Z)=0$$,
for all $Y,Z$ in $TM$. This eventually implies
$$\mathcal{L}_{\xi}(\mathcal{L}_{X}g)(Y,Z)-\mathcal{L}_{X}g([\xi ,Y],Z)-\mathcal{L}_{X}g(Y,[\xi ,Z])=0.$$
Then after some calculations and using the condition $(\nabla _Xg)(\xi ,W)=0$ we get
\begin{equation}
R(X,\xi,\xi,Y)+\nabla_Yg(\nabla_\xi X,\xi)+g(\nabla_\xi\nabla_\xi X,Y)=0.
\end{equation}
Now, as the vector field $Y$ is orthogonal to $\xi$, using the condition (2.4) in result (2.2) we get
\begin{equation}
R(X,\xi,\xi,Y)=g(\xi,\xi)g(X,Y)-g(X,\xi)g(\xi,Y)=g(X,Y).
\end{equation}
So, combining equations (3.2) and (3.3) we have
$$g(X,Y)+\nabla_Yg(\nabla_\xi X,\xi)+g(\nabla_\xi\nabla_\xi X,Y)=0.$$
Also from equation (3.1) using the value of $g(\nabla_\xi X,\xi)$ in above we obtain
$$g(X,Y)+\nabla_Y(\lambda -\frac{p}{2}+\frac{k}{2})g(\xi ,\xi)+g(\nabla_\xi\nabla_\xi X,Y)=0.$$
As $k$ is a constant and $\eta(\xi)=1$ , the above equation implies
\begin{equation}
g(X,Y)+g(\nabla_\xi\nabla_\xi X,Y)+\nabla_Y(\lambda -\frac{p}{2})=0.
\end{equation}
Again the soliton being gradient almost Ricci soliton, putting $X=\nabla f$, equation (3.4) becomes
\begin{equation}
g(\nabla f,Y)+g(\nabla_\xi\nabla_\xi\nabla f,Y)+\nabla_Y(\lambda -\frac{p}{2})=0.
\end{equation}
\par
Next, observe that from equation (3.1) we can write $\nabla_\xi X=(\lambda -\frac{p}{2}+\frac{k}{2})\xi$ and using the gradient almost soliton condition $X=\nabla f$ in this, we have
\begin{equation}
\nabla_\xi\nabla f=(\lambda -\frac{p}{2}+\frac{k}{2})\xi.
\end{equation}
So, now equations(3.5) and (3.6) together yield
\begin{equation}
g(\nabla f,Y)+g(\nabla_\xi(\lambda -\frac{p}{2}+\frac{k}{2})\xi,Y)+\nabla_Y(\lambda -\frac{p}{2})=0.
\end{equation}
Again we know that, $\nabla_\xi(\lambda -\frac{p}{2}+\frac{k}{2})\xi=(\lambda -\frac{p}{2}+\frac{k}{2})\nabla_\xi\xi+(\xi(\lambda -\frac{p}{2}+\frac{k}{2}))\xi$.
 And as $\nabla_\xi\xi=0$ we have $\nabla_\xi(\lambda -\frac{p}{2}+\frac{k}{2})\xi=(\xi(\lambda -\frac{p}{2}+\frac{k}{2}))\xi$. So now putting this value in equation (3.7) we obtain
\begin{equation}
g(\nabla f,Y)+\xi(\lambda -\frac{p}{2}+\frac{k}{2})g(Y,\xi)+\nabla_Y(\lambda -\frac{p}{2})=0,
\end{equation}
for any vector field $Y$ in the distribution $\mathcal{D}\subset TM$. Thus the result is true for all vector field $Y$ satisfying $g(Y,\xi)=0$ and hence we have from the last equation
\begin{equation}
g(\nabla f,Y)=Y(\frac{p}{2}-\lambda).
\end{equation}
\par
\medskip
Now, if we consider $\lambda=\frac{p}{2}$, i.e; if $\lambda$ is half of the non-dynamical conformal pressue field $p$; we get $g(\nabla f,Y)=0$  $\forall Y\in\mathcal{D}\subset TM$. Hence we can conclude that the potential function $f$ is constant.
\end{proof}\par
\medskip
\medskip
Next, we consider an almost conformal gradient Ricci soliton, on a Sasakian manifold, without the distribution $\mathcal{D}\subset TM$ used in the previous theorem. Here we try to establish a relation between the smooth function $\lambda$ and the potential function $f$ of the soliton. But before going into our main result let us first establish a lemma which will be required later.
\medskip
\begin{lem}
Let $(M^{2n+1},g)$ be a Sasakian manifold satisfying the almost conformal gradient Ricci soliton equation, then the curvature tensor $R(X,Y)Z$ of the manifold satisfies the following relation,
  \begin{equation}
    R(X,Y)Df=(X\lambda)Y-(Y\lambda)X-(\nabla_XQ)Y+(\nabla_YQ)X,
  \end{equation}
  where $Df=grad f$, for some smooth function $f$.
\end{lem}
\begin{proof}
  Assuming the manifold $(M^{2n+1},g)$ satisfies the almost conformal gradient Ricci soliton equation, from $(1.8)$ we can write
  \begin{equation}
    \mathcal{L}_{Df}g(X,Y)+2Ric(X,Y)=[2\lambda -(p+\frac{2}{2n+1})]g(X,Y),
  \end{equation}
  where $Df=grad f$, is the soliton vector field in this case, with some smooth function $f$.
   Again it is well known from the definition of Lie derivative that, $\mathcal{L}_{V}g(X,Y)=g(\nabla_XV,Y)+g(X,\nabla_YV)$ and using this formula the above equation $(3.11)$ becomes,
  \begin{equation}
    Hess^f(X,Y)=-Ric(X,Y)+[\lambda -(\frac{p}{2}+\frac{1}{2n+1})]g(X,Y),
  \end{equation}
  where $Hess^f(X,Y)$ is the Hessian operator of the function $f$.\\
   Now from the above equation $(3.12)$ it can be written that,
  \begin{equation}
    \nabla_XDf=-QX+[\lambda -(\frac{p}{2}+\frac{1}{2n+1})]X,
  \end{equation}
  where $Q$ is the Ricci operator given by $g(QX,Y)=Ric(X,Y)$.\\
  Differentiating, both sides of the above equation $(3.13)$, covariantly in the direction of an arbitrary vector field $Y$ we get,
  \begin{equation}
    \nabla_Y\nabla_XDf=-(\nabla_YQ)X-Q(\nabla_YX)+[\lambda -(\frac{p}{2}+\frac{1}{2n+1})]\nabla_YX+(Y\lambda)X.
  \end{equation}
  Swapping the variables $X$ and $Y$ in the previous equation $(3.14)$ yields,
  \begin{equation}
    \nabla_X\nabla_YDf=-(\nabla_XQ)Y-Q(\nabla_XY)+[\lambda -(\frac{p}{2}+\frac{1}{2n+1})]\nabla_XY+(X\lambda)Y.
  \end{equation}
  Also from equation $(3.13)$ replacing $X$ by $[X,Y]$ we have,
   \begin{equation}
    \nabla_{[X,Y]}Df=-Q[X,Y]+[\lambda -(\frac{p}{2}+\frac{1}{2n+1})][X,Y],
  \end{equation}
  where $[X,Y]$ denotes the usual Lie bracket operation.\\
  Now using equations $(3.14)$, $(3.15)$ and $(3.16)$ in the well known formula $R(X,Y)Z=\nabla_X\nabla_Y Z-\nabla_Y\nabla_X Z-\nabla_{[X,Y]}Z$ of Riemannian curvature tensor and after simplification finally we get,
  \begin{equation}
    R(X,Y)Df=(X\lambda)Y-(Y\lambda)X-(\nabla_XQ)Y+(\nabla_YQ)X,\nonumber
  \end{equation}
    which is the desired result and hence completes the proof.
\end{proof}
\medskip
Now we are in a position to establish our next theorem of this section which is as follows:
\medskip
\begin{thm}
Let $(M^{2n+1},g)$ be a Sasakian manifold satisfying the almost conformal gradient Ricci soliton equation, then $(\lambda+f)$ is constant; where $f$ is the potential function of the soliton.
\end{thm}
\begin{proof}
  Let us consider that a Sasakian manifold which satisfies the almost conformal gradient Ricci soliton equation, with the potential function $f$. Then from the  equation $(3.10)$ of the lemma $3.2$ the Riemannian curvature tensor of the manifold satisfies the relation
  \begin{equation}
    R(X,Y)Df=(X\lambda)Y-(Y\lambda)X-(\nabla_XQ)Y+(\nabla_YQ)X.
  \end{equation}
  Taking inner product on both sides of equation $(3.17)$ with an arbitrary vector field $Z$ gives us,
  \begin{eqnarray}
    g(R(X,Y)Df,Z) &=& (X\lambda)g(Y,Z)-(Y\lambda)g(X,Z) \nonumber\\
     &&-(\nabla_XQ)g(Y,Z)+(\nabla_YQ)g(X,Z).
  \end{eqnarray}
  Now since from the symmetry of $R$ we know that $g(R(X,Y)Df,Z)=\\
  -g(R(X,Y)Z,Df)$ and then from the above equation $(3.18)$ we get,
  \begin{eqnarray}
    g(R(X,Y)Z,Df) &=& (\nabla_XQ)g(Y,Z)-(\nabla_YQ)g(X,Z) \nonumber\\
     &&-(X\lambda)g(Y,Z)+(Y\lambda)g(X,Z).
  \end{eqnarray}
  Putting $Z=\xi$ in equation $(3.19)$ we have,
  \begin{eqnarray}
    g(R(X,Y)\xi,Df) &=& (\nabla_XQ)\eta(Y)-(\nabla_YQ)\eta(X) \nonumber\\
     &&-(X\lambda)\eta(Y)+(Y\lambda)\eta(X),
  \end{eqnarray}
   $\forall X,Y\in TM$, where $TM$ is the tangent bundle of the manifold.\\
  Again as the given manifold is Sasakian, from equation $(2.4)$ it can be easily obtained that, the curvature tensor sastisfies $R(X,Y)\xi=\eta(Y)X-\eta(X)Y, \forall X,Y\in TM$. Then taking an inner product of the resulting equation with the vector field $Df$ we can obtain,
 \begin{equation}
    g(R(X,Y)\xi,Df)=\eta(Y)(Xf)-\eta(X)(Yf),
  \end{equation}
 $\forall X,Y\in TM$. Now combining the equations $(3.20)$ and $(3.21)$ and after simplification we get,
 \begin{equation}
   X(\lambda+f)[\eta(Y)-\eta(X)]=(\nabla_XQ)\eta(Y)-(\nabla_YQ)\eta(X).
 \end{equation}
 Interchanging $X$ and $Y$ in $(3.22)$ yields,
 \begin{equation}
   Y(\lambda+f)[\eta(X)-\eta(Y)]=(\nabla_YQ)\eta(X)-(\nabla_XQ)\eta(Y).
 \end{equation}
 Finally, adding the equations $(3.22)$ and $(3.23)$ we arrive at,
  \begin{equation}
   [X(\lambda+f)-Y(\lambda+f)][\eta(Y)-\eta(X)]=0,
 \end{equation}
 $\forall X,Y\in TM$. Again as $\eta(X)\neq\eta(Y)$ $\forall X,Y\in TM$, from the above equation $(3.24)$ we can conclude that $D(\lambda+f)=0$. Therefore $(\lambda+f)$ is constant and hence completes the proof.
\end{proof}
\par
\medskip
In particular if we take $\lambda$ to be constant, i.e; for the case of conformal gradient Ricci soliton, from the above we get $f$ is constant
\begin{cor}
  Let $(M^{2n+1},g)$ be a Sasakian manifold satisfying the conformal gradient Ricci soliton equation, then the potential function $f$ of the soliton is a constant function.
\end{cor}
\par
\medskip
Now we prove a sufficient condition for an almost conformal Ricci soliton to be an almost conformal gradient Ricci soliton using some condition on the potential vector field. So, for this purpose we need a special kind of vector field, called concurrent vector field \cite{ChD}, which is defined as follows
\begin{defn}
  A vector field $V$ on a Riemannian $n$-manifold $M^n$ is said to be a concurrent vector field if it satisfies the equation
  \begin{equation}
  \nabla_XV=X,
  \end{equation}
  for all smooth vector fields $X$ on $M^n$.
\end{defn}
\par
\medskip
K. Yano and B. Y. Chen \cite{YC} in 1971 proved that, if the holonomy group of a Riemannian $n$-manifold $M^n$ leaves a point invariant, then there exists a concurrent vector field $V$ on $M^n$. Recently, in \cite{ChD}, the authors gave a classification of Ricci solitons on a Riemannian manifold equipped with concurrent potential vector field. Subsequently, many mathematicians have studied concurrent vector fields in Riemannian geometry. Now we will discuss our final result as follows
\begin{thm}
Let $(M^{2n+1},g)$ be a Sasakian manifold admitting an almost conformal Ricci soliton $(V,\lambda)$, with concurrent potential vector field $V$, then the soliton becomes an almost conformal gradient Ricci soliton and the manifold becomes an Einstein manifold. Moreover, the soliton is shrinking, steady or expanding according as the conformal pressure $p$ satisfies the relations $(mp+2m^2+2)>0$, $(mp+2m^2+2)=0$ or $(mp+2m^2+2)<0$; where $m=(2n+1)$ is the dimension of the manifold.
\end{thm}
\begin{proof}
Let us consider a Sasakian manifold $(M^{2n+1},g)$ admitting an almost conformal Ricci soliton $(V,\lambda)$, with concurrent potential vector field $V$, then we have
\begin{equation}
(\mathcal{L}_{V}g)(X,Y)+2Ric(X,Y)=[2\lambda -(p+\frac{2}{2n+1})]g(X,Y),
\end{equation}
for all smooth vector fields $X,Y$ on $M^{2n+1}$. Again since the vector field $V$ is concurrent, so it satisfies the equation $(3.25)$. Therefore recalling the definition of Lie derivative and in view of equation $(3.25)$ we deduce that
\begin{eqnarray}
(\mathcal{L}_{V}g)(X,Y) &=& g(\nabla_XV,Y)+g(X,\nabla_YV) \nonumber\\
                        &=& 2g(X,Y).
\end{eqnarray}
Now plugging in the value of $\mathcal{L}_{V}g$ from $(3.27)$, the equation $(3.26)$ reduces to
\begin{equation}
Ric(X,Y)=[(\lambda -1)-(\frac{p}{2}+\frac{1}{2n+1})]g(X,Y),
\end{equation}
for all smooth vector fields $X,Y$ on $M^{2n+1}$. This proves that the manifold becomes an Einstein manifold.
\par
\medskip
Further, if we set, $f:=\frac{1}{2}g(V,V)$; then one can easily compute that;
\begin{equation}
Hess^f(X,Y)=g(X,Y),
\end{equation}
for all smooth vector fields $X,Y$ on $M^{2n+1}$. Therefore in view of the equations $(3.28)$ and $(3.29)$ we can conclude that
\begin{equation}
Hess^f(X,Y)+Ric(X,Y)=[\lambda-(\frac{p}{2}+\frac{1}{2n+1})]g(X,Y),
\end{equation}
for all smooth vector fields $X,Y$ on $M^{2n+1}$. This proves that the soliton is an almost conformal gradient Ricci soliton, with the above defined function $f$ as the potential function.
\par
\medskip
Now let us consider an orthonormal basis of the Sasakian manifold $M^{2n+1}$ as $\{e_i:1\leq i\leq(2n+1)\}$. Taking an inner-product of the equation $(2.4)$ with an arbitrary vector field $Y$, then setting $Y=e_i$ and summing over $1\leq i\leq(2n+1)$, we obtain
\begin{equation}
Ric(X,\xi)=2n\eta(X).
\end{equation}
Again taking $Y=\xi$ in the equation $(3.28)$ yields
\begin{equation}
Ric(X,\xi)=[(\lambda -1)-(\frac{p}{2}+\frac{1}{2n+1})]\eta(X).
\end{equation}
Thus comparing $(3.31)$ and $(3.32)$ we finally arrive at $\lambda=(2n+1)+(\frac{p}{2}+\frac{1}{2n+1})$. Hence from this value of $\lambda$ and the fact that the soliton is shrinking, steady or, expanding according as; $\lambda >0$, $\lambda =0$ or, $\lambda <0$; completes the proof.
\end{proof}
\par
\medskip
\textbf{Example:} Let us consider the $5$-dimensional manifold $M=\{(u_1,u_2,v_1,v_2,w)\in\mathbb{R}^5\}$. Define a set of vector fields
$\{e_i: 1\leq i\leq 5\}$ on the manifold $M$ given by
\begin{equation}
e_1=2(u_2\frac{\partial}{\partial v_1}-\frac{\partial}{\partial u_1}), e_2=\frac{\partial}{\partial u_2}, e_3=-2\frac{\partial}{\partial v_1}, e_4=2(w\frac{\partial}{\partial v_1}-\frac{\partial}{\partial v_2}), e_5=-2\frac{\partial}{\partial w}.\nonumber
\end{equation}
Let us define the Riemannian metric $g$ on $M$ by
\begin{equation}
g(e_i,e_j)=\left\{ \begin{array}{rcl}
1, & \mbox{for}
& i=j \\ 0, & \mbox{for} & i\neq j
\end{array}\right.\nonumber
\end{equation}
for all $i,j=1,2,3,4,5$. Now considering $e_5=\xi$, let us take the $1$-form $\eta$, on the manifold $M$, defined by
\begin{equation}
\eta(U)=g(U,e_3)=g(U,\xi),~~~\forall U\in TM.\nonumber
\end{equation}
Then it can be observed that $\eta(e_3)=1$. Let us define the $(1,1)$ tensor field $\phi$ on $M$ as
\begin{equation}
\phi(e_1)=e_2, ~~\phi(e_2)=-e_1,~~\phi(e_3)=0,~~e_4=e_5,~~\phi(e_5)=-e_4.\nonumber
\end{equation}
Then using the linearity of $g$ and $\phi$ it can be easily checked that
\begin{equation}
\phi^2(U)=-U+\eta(U)\xi,~~g(\phi U,\phi V)=g(U,V)-\eta(U)\eta(V),~~~\forall U,V\in TM.\nonumber
\end{equation}
Hence the structure $(\phi, \xi, \eta, g)$ defines an almost contact structure on $M$.\\
Now, using the definitions of Lie bracket, direct computations give us \\
$[e_1,e_2]=[e_4,e_5]=2e_3$ and for all other $i,j$'s the term $[e_i,e_j]$ vanishes.\\
Again the Riemannian connection $\nabla$ of the metric $g$ is defined by the well-known Koszul's formula which is given by
\begin{eqnarray}
  2g(\nabla_XY,Z) &=& Xg(Y,Z)+Yg(Z,X)-Zg(X,Y) \nonumber\\
   && -g(X,[Y,Z])+g(Y,[Z,X])+g(Z,[X,Y]).\nonumber
\end{eqnarray}
Using the above formula one can easily calculate that \\
$\nabla_{e_1}e_2=\nabla_{e_4}e_5=e_3$, $\nabla_{e_1}e_3=\nabla_{e_3}e_1=-e_2$, $\nabla_{e_2}e_1=-e_3$, $\nabla_{e_2}e_3=\nabla_{e_1}e_2=e_1$, $\nabla_{e_3}e_5=e_4$, $\nabla_{e_4}e_3=-e_5$ and for all other $i,j$'s the term $\nabla_{e_i}e_j$ vanishes.\\
 Thus it follows that $(\nabla_X\phi)Y=g(X,Y)\xi-\eta(Y)X,~~~\forall X,Y\in TM$. Therefore the manifold $(M,g)$ is a $5$-dimensional Sasakian manifold.\\
Now using the well-known formula $R(X,Y)Z=\nabla_X\nabla_YZ-\nabla_Y\nabla_XZ-\nabla_{[X,Y]}Z$ the non-vanishing components of the Riemannian curvature tensor $R$ can be easily obtained as
\begin{eqnarray}
  R(e_1,e_2)e_1 &=& 3e_2, R(e_2,e_3)e_1=-e_3, R(e_2,e_3)e_2=-e_3, R(e_2,e_4)e_1=-e_5, \nonumber\\
  R(e_1,e_2)e_2 &=& -e_1, R(e_1,e_2)e_5=-2e_4, R(e_2,e_4)e_5=e_1, R(e_2,e_5)e_1=e_4, \nonumber\\
  R(e_1,e_3)e_1 &=& -e_3, R(e_1,e_3)e_3=e_1, R(e_3,e_4)e_3=-e_4, R(e_4,e_5)e_1=2e_2, \nonumber\\
  R(e_1,e_4)e_2 &=& e_5, R(e_1,e_4)e_5=-e_2, R(e_4,e_5)e_2=-2e_1, R(e_4,e_5)e_4=2e_5, \nonumber\\
  R(e_1,e_5)e_4 &=& e_2, R(e_4,e_5)e_5=-2e_4.\nonumber
\end{eqnarray}
From the above values of the curvature tensor, we obtain the components of the Ricci tensor as follows
\begin{equation}
S(e_1,e_1)=-2, S(e_2,e_2)=3, S(e_3,e_3)=4, S(e_4,e_4)=4, S(e_5,e_5)=-1.
\end{equation}
Therefore taking $X=e_3=\xi$ in the conformal Ricci soliton equation $(1.8)$ and then tracing it out, in view of $(3.33)$ we can calculate $\lambda=(\frac{p}{2}+\frac{9}{5})$. Hence we can conclude that for $\lambda=(\frac{p}{2}+\frac{9}{5})$ the data $(\xi,\lambda)$ defines an almost conformal Ricci soliton on the $5$-dimensional Sasakian manifold $(M,g,\phi,\xi,\eta)$.
\par
\medskip
\medskip
\textbf{Conclusion and remarks:} Soliton solutions are a special class of solutions which play an important role to study the singularities of geometric flows appearing as possible models of singularity. Here we have proved an important condition for the potential vector field of the soliton to be constant in terms of the non-dynamical scalar field. In conformal geometry, conformal Ricci solitons play a vital role as they are natural generalization of the Einstein metric. Moreover, presently the geometric flows are the most effective means to study the geometry of relativistic perfect fluid spacetime. In this paper, the effect of almost conformal gradient Ricci solitons have been studied within the framework of Sasakian manifolds, which is a very important class of contact manifold having extensive use in relativity and mathematical physics. Hence it is interesting to investigate conformal Ricci soliton and various other geometric flows like Yamabe flow, on Sasakian manifolds as well as in other contact and complex manifold structures.\par
\medskip
\medskip
\medskip
\medskip
\textbf{Acknowledgement:} The first author D. Ganguly is thankful to the National Board for Higher Mathematics (NBHM),India, for their financial support(Ref No: 0203/11/2017/RD-II/10440) to carry on this research work.

\end{document}